\documentclass[12pt]{article}

\usepackage{datetime}

\usepackage{eucal}

\usepackage[all]{xy}

\usepackage{tikz}

\usepackage{amssymb, amsmath, amsthm, eucal, stmaryrd}

\usepackage{mathptmx}
\usepackage[scaled=.90]{helvet}
\usepackage{courier}

\usepackage[pdftex]{hyperref}

\hypersetup{
  colorlinks = true,
  linkcolor = black,
  urlcolor =  blue,
  citecolor = black,
}

\newcommand{\Sets}{\mathbf{Sets}}
\newcommand{\Perm}{\mathbf{Perm}}
\newcommand{\Rack}{\mathbf{Rack}}
\newcommand{\Quan}{\mathbf{Quan}}
\newcommand{\Invo}{\mathbf{Invo}}
\newcommand{\Kei}{\mathbf{Kei}}

\newcommand{\bfC}{\mathbf{C}}

\newcommand{\bbZ}{\mathbb{Z}}

\newcommand{\rmF}{\mathrm{F}}

\renewcommand{\phi}{\varphi}

\DeclareMathOperator{\can}{can}
\DeclareMathOperator{\per}{per}
\DeclareMathOperator{\id}{id}
\DeclareMathOperator{\Id}{Id}

\newtheorem{theorem}{Theorem}
\newtheorem{proposition}[theorem]{Proposition}

\newtheorem{lemma}[theorem]{Lemma}

\theoremstyle{definition}
\newtheorem{definition}[theorem]{Definition}
\newtheorem{example}[theorem]{Example}

\newtheorem{remark}[theorem]{Remark}

\numberwithin{theorem}{section}
\numberwithin{equation}{section}
\numberwithin{figure}{section}

\title{Permutations, power operations, and the \hbox{center of the category of racks}}

\author{Markus Szymik}

\newdateformat{mydate}{\monthname~\twodigit{\THEYEAR}}
\date{\mydate\today}

\begin{document}

\maketitle

\renewcommand{\abstractname}{\vspace{-2\baselineskip}}

\begin{abstract}
\noindent %Abstract:
Racks and quandles are rich algebraic structures that are strong enough to classify knots. Here we develop several fundamental categorical aspects of the theories of racks and quandles and their relation to the theory of permutations. In particular, we compute the centers of the categories and describe power operations on them, thereby revealing free extra structure that is not apparent from the definitions. This also leads to precise characterizations of these theories in the form of universal properties.

\vspace{\baselineskip}
\noindent MSC: 
57M27,	% Invariants of knots and 3-manifolds
20N02, 	% Sets with a single binary operation
18C10. % Theories (e.g. algebraic theories), structure, and semantics

\vspace{\baselineskip}
\noindent Keywords: 
Permutations, racks, quandles, centers, power operations, algebraic theories.
\end{abstract}

%%%

\section*{Introduction}

Racks and quandles are algebraic structures that are directly related to the topo\-logy and geo\-metry of braids and knots. See the original sources~\cite{Joyce},~\cite{Matveev}, \cite{Brieskorn}, and~\cite{Fenn+Rourke}, as well as the more recent introductions~\cite{Nelson:2011},~\cite{Elhamdadi+Nelson}, and~\cite{Nelson:2016}. In this paper we develop several fundamental categorical aspects of the theories of racks and quandles. This pursues the general goal to raise our equation-based understanding of these structures to a more conceptual and functorial level.

Given any category one can ask for the symmetries that all of its objects have in common. This led Bass, Mac~Lane, and others to the notion of the center of a category. It is not extraordinary that the center of a category is trivial (as for the category of sets) or non-trivial but uninteresting~(as for the category of groups). However, there are exceptions. For instance, the center of the category of commutative rings in prime characteristic is the abelian monoid freely generated by Frobenius, and the ubiquitous Frobenius actions derived from it exploit this extra symmetry that comes for free. One of the results that we prove here~(Theorem~\ref{thm:center}) states that the category of racks is similarly rich: Its center is the free abelian group generated by the canonical automorphism~$\rmF_R\colon R\to R$ that lives on every rack~$R$. 

The center of a category comprises the symmetries of the identity functor. In the case of racks and quandles there are other endofunctors than the identity that are of interest, and that are presented here next: the power operations~$\Psi^m$ for integers~$m$, see Theorem~\ref{thm:powerops}. Their existence is vaguely inspired by the Adams operations in topological and algebraic~K-theory. It is possible that there be more such operations than the ones introduced here, see Remark~\ref{rem:Kan}, but already these power operations can be used to give precise characterizations of the algebraic theories of quandles, involutary racks, and kei in terms of the more fundamental theories of racks and permutations. For instance, Theorem~\ref{thm:char_quan} asserts the existence of a pushout square
\[
\xymatrix{
\Perm\ar[r]\ar[d]&\Rack\ar[d]&\\
\Sets\ar[r]&\Quan
}
\]
of algebraic theories, and all arrows involved in this diagram are split (according to Propositions~\ref{prop:for_splitting_R},~\ref{prop:Q_splits_off_S}, and~\ref{prop:q_retract_of_r}). See Theorems~\ref{thm:char2} and~\ref{thm:char_kei} for other results of a similar flavor.

The outline of this paper is as follows. In Section~\ref{sec:canonical} we discuss the canonical automorphisms of racks. This is already used in the following Sections~\ref{sec:perm} and~\ref{sec:quan} to give a preliminary discussion of the relation between racks on the one hand, and permutations and quandles on the other. Then we will prove in Section~\ref{sec:centers} that the center of the category of racks is the free cyclic group that is generated by the canonical automorphism. Section~\ref{sec:powers} discusses power operations available for permutations, racks, and quandles, and how they relate to each other. The final Section~\ref{sec:universal} uses all of this to give definitive formulations of the relations between the various theories in these terms.

%%%

\section{The canonical automorphism of a rack}\label{sec:canonical}

There are different notational conventions when it comes to racks. Here is the one that we will be using in this writing. 

\begin{definition}
A {\em rack}~$(R,\rhd)$ is a set~$R$ together with a binary operation~$\rhd$ such that all left multiplications
\[
\ell_x\colon R\longrightarrow R,\,y\longmapsto x\rhd y=\ell_x(y)
\]
are automorphisms, i.e.~they are bijective and satisfy~$\ell_x(y\rhd z)=\ell_x(y)\rhd\ell_x(z)$, or
\[
x\rhd(y\rhd z)=(x\rhd y)\rhd(x\rhd z).
\]
\end{definition}

We start with some elementary observations.

\begin{lemma}\label{lem:sur}
Every element~$y$ in a rack can be written in the form~$y=x\rhd x$ for some element~$x$.
\end{lemma}

\begin{proof}
Given an element~$y$ in a rack~$R$, there is a (unique) element~$x$ in the rack~$R$ such that~$y\rhd x= y$. For this element~$x$ we have
\[
y\rhd(x\rhd x)=(y\rhd x)\rhd(y\rhd x)=y\rhd y.
\]
Since~$?\mapsto y\,\rhd ?$ is a bijection, this implies~$x\rhd x=y$.
\end{proof}

\begin{lemma}\label{lem:rel}
In any rack, we have the relation
\[
(x\rhd x)\rhd y=x\rhd y
\]
for all elements~$x$ and~$y$.
\end{lemma}

\begin{proof}
To see this, let~$z$ be the element such that~$x\rhd z=y$. Then
\[
(x\rhd x)\rhd y=(x\rhd x)\rhd(x\rhd z)=x\rhd(x\rhd z)=x\rhd y,
\]
as claimed.
\end{proof}

\begin{lemma}\label{lem:bij}
Let~$(R,\rhd)$ be a rack. Then the composition
\[
\rmF\colon
R\overset{(\id,\id)}{\xrightarrow{\hspace*{3em}}} R\times R\overset{\rhd}{\xrightarrow{\hspace*{3em}}} R%,\,x\mapsto x\rhd x
\]
that sends~$x$ to~$x\rhd x$ is a bijection.
\end{lemma}

\begin{proof}
The map is surjective by Lemma~\ref{lem:sur}. For injectivity, we have to show that the equation~$x\rhd x=y$ determines~$x$. But that equation implies, using Lemma~\ref{lem:rel}, that we have
\[
y\rhd x=(x\rhd x)\rhd x=x\rhd x=y,
\]
and this indeed determines~$x$ uniquely.
\end{proof}

We can improve the statement of Lemma~\ref{lem:bij}:

\begin{proposition}\label{prop:aut}
For all racks~$(R,\rhd)$ the bijection~$\rmF\colon R\to R$ from Lemma~\ref{lem:bij},
\[
\rmF(x)=x\rhd x,
\]
is an automorphism of the rack~$(R,\rhd)$.
\end{proposition}

\begin{proof}
We rewrite the relation from Lemma~\ref{lem:rel} in the form
\begin{equation}\label{eq:rel_sigma}
\rmF(x)\rhd y=x\rhd y
\end{equation}
for all~$x$ and~$y$. We can then calculate:
\begin{equation}\label{eq:act}
\rmF(x\rhd y)=(x\rhd y)\rhd(x\rhd y)=x\rhd(y\rhd y)=x\rhd\rmF(y).
\end{equation}
Together with~\eqref{eq:rel_sigma} we then get
\[
\rmF(x\rhd y)=\rmF(x)\rhd\rmF(y).
\]
and this finishes the proof.
\end{proof}

It will be convenient to single out these automorphisms:

\begin{definition}
If~$(R,\rhd$) is a rack, its automorphism~$\rmF$ from Proposition~\ref{prop:aut} will be referred to as its {\it canonical automorphism}. We will sometimes write~$\rmF_R$ or~$\rmF_\rhd$ for clarity.
\end{definition}

\begin{remark}
The canonical automorphism $F$ is the inverse of the map that Bries\-korn~\cite[Sec.~2]{Brieskorn} denotes by~$\iota$. See also~\cite[Sec.~1.1.1]{Andruskiewitsch+Grana}. One advantage of~$F$ over its inverse is that it can be defined by the explicit formula $F(x)=x\rhd x$, rather than implicitly by the equation~$x\rhd\iota(x)=x$.
\end{remark}

%%%

\section{Splitting off permutations}\label{sec:perm}

For the purposes of the present text, a {\em permutation} is a set~$S$ together with a bijection~$f\colon S\to S$. A morphism~$(S,f)\to(T,g)$ of permutations is a map~$\phi\colon S\to T$ of sets that commutes with the permutations, so that~$\phi f=g\phi$. We have an algebraic theory~$\Perm$ of permutations in the sense of Lawvere~\cite{Lawvere}. This means, in particular, that the forgetful functor~$(S,f)\mapsto S$, being representable by the free object on one generator, has a left adjoint for abstract reasons~\cite{Freyd}. In this case, it is easy to give an explicit model for the left adjoint, the `free permutation' functor. The free permutation on a set~$B$ can be modeled on the set~$\bbZ\times B$ with the bijection~\hbox{$(n,b)\mapsto(n+1,b)$}. We will write morphisms between algebraic theories in the direction of the left adjoint, so that the forgetful-free adjunction just described gives the (unique) structure morphism
\[
\Sets\longrightarrow\Perm
\]
of the permutation theory. There is also a morphism
\[
\Perm\overset{\id}{\longrightarrow}\Sets
\]
of algebraic theories, where the right-adjoint equips a set~$S$ with the bijection~$\id_S$. These two morphisms form a section-retraction pair, so that the composition
\[
\Sets\longrightarrow\Perm\overset{\id}{\longrightarrow}\Sets
\]
is the identity. Note that it is sufficient to check that for the right adjoints, and it follows for the left-adjoints.

\begin{remark}\label{rem:actions}
Given any monoid~$M$, there is an algebraic theory of~$M$-sets (sets with an action of~$M$). The theory of permutations just described is the special case when~$M=\mathbb{Z}$ is the additive monoid (a group, in fact) of integers.
\end{remark}

We can now see that the theory of permutations splits off of the theory of racks as a retract.

\begin{proposition}\label{prop:for_splitting_R}
There are morphisms
\begin{gather*}
\Perm\overset{\can}{\xrightarrow{\hspace*{2em}}}\Rack\overset{\per}{\xrightarrow{\hspace*{2em}}}\Perm
\end{gather*}
of algebraic theories whose composition is the identity.
\end{proposition}

\begin{proof}
Lemma~\ref{lem:bij} allows us to define a `forgetful' (i.e. right-adjoint) functor
\[
(R,\rhd)\longmapsto(R,\rmF)
\]
from~$\Rack$ to~$\Perm$ that sends a rack to the underlying set together with the permutation given by the canonical automorphism. This is a morphism
\[
\Perm\overset{\can}{\xrightarrow{\hspace*{2em}}}\Rack
\]
of algebraic theories.

There is also a `forgetful' (i.e. right-adjoint) functor from~$\Perm$ to~$\Rack$. It sends a permutation~$(S,f)$ to the rack~$(S,\rhd_f)$, where
\[
x\rhd_f y = f(y).
\]
The calculations
\[
x\rhd_f (y\rhd_f z) = x\rhd_ff(z)=f^2(z)
\]
and
\[
(x\rhd_f y)\rhd(x\rhd_f z) = f(y)\rhd_ff(z)=f^2(z)
\]
show that this indeed defines a rack. There is a corresponding morphism
\[
\Rack\overset{\per}{\xrightarrow{\hspace*{2em}}}\Perm
\]
of algebraic theories. 

It is straightforward to check that the composition of the right-adjoints is the identity functor, and it follows for the left-adjoints.
\end{proof}

%%%

\section{Splitting off quandles}\label{sec:quan}

Using the terminology introduced above, we have the following definition.

\begin{definition}
A rack is a {\em quandle} if its canonical automorphism is the identity.
\end{definition}

Since~$\rmF(x)=x\rhd x$, we have~$\rmF=\id$ if and only if~$x\rhd x=x$ for all~$x$. Therefore, this agrees with the usual definition.

\begin{proposition}\label{prop:Q_splits_off_S}
There are morphisms
\begin{gather*}
\Sets\longrightarrow\Quan\longrightarrow\Sets
\end{gather*}
of algebraic theories whose composition is the identity.
\end{proposition}

\begin{proof}
This is analogous to the proof of Proposition~\ref{prop:for_splitting_R}, just simpler. The morphism on the left is the structure morphism of the algebraic theory of quandles, and the morphisms on the right is given by (i.e.~the right-adjoint is) the trivial quandle structure~\hbox{$x\rhd y=y$} on any given set.
\end{proof}

While the definition of quandles might suggest that the construction~\hbox{$\rhd\mapsto\rmF_\rhd$} is uninteresting for those who are only interested in quandles, the contrary is the case: It can be used to turn any rack into a quandle. 

\begin{proposition}\label{prop:box}
Let~$(R,\rhd)$ be a rack with canonical automorphism~$\rmF$. Then
\[
x\boxempty y = \rmF^{-1}(x\rhd y)
\]
is a quandle structure on~$R$.
\end{proposition}

\begin{remark}
We can also write
\begin{equation}\label{eq:also}
x\boxempty y =x\rhd \rmF^{-1}(y)
\end{equation}
in view of~\eqref{eq:act}. 
\end{remark}

\begin{proof}
It is straightforward to verify that~$\boxempty$ is self-distributive: On the one hand, we have
\begin{align*}
x\boxempty(y\boxempty z)
&=x\rhd\rmF^{-1}(y\rhd\rmF^{-1}z)\\
&=x\rhd(y\rhd\rmF^{-2}z).
\end{align*}
On the other hand, we also get
\begin{align*}
(x\boxempty y)\boxempty(x\boxempty z)
&=(x\rhd\rmF^{-1}y)\rhd\rmF^{-1}(x\rhd\rmF^{-1}z)\\
&=(x\rhd\rmF^{-1}y)\rhd(x\rhd\rmF^{-2}z)\\
&=x\rhd(\rmF^{-1}y\rhd\rmF^{-2}z)\\
&=x\rhd(y\rhd\rmF^{-2}z)
\end{align*}
using~\eqref{eq:rel_sigma} again. This shows that the maps~$\ell_x^\boxempty=y\mapsto x\boxempty y$ are homomorphisms. In addition, for every element~$x$ in~$R$, the map~$\ell_x^\boxempty$ is the composition of the bijection~$\ell_x^\rhd$ with the bijection~$\rmF^{-1}$, hence bijective. In other words,~$(R,\boxempty)$ is a rack.

Lastly, we have
\[
x\boxempty x = \rmF^{-1}(x\rhd x)=\rmF^{-1}\rmF(x)=x,
\]
so that this rack is indeed a quandle, as claimed.
\end{proof}

\begin{remark}\label{rem:comp}
If the rack~$(R,\rhd)$ is already a quandle, then the canonical automorphism~$\rmF$ is the identity, and we have~$\boxempty=\rhd$.
\end{remark}

So far we have considered only individual racks and quandles, as in the references~\cite[Sec.~2]{Brieskorn} and~\cite[Sec.~1.1.1]{Andruskiewitsch+Grana} before. We are now going to enhance the statements by passing to algebraic theories: Proposition~\ref{prop:box} shows that there is a~`forgetful' functor
\[
(R,\rhd)\longmapsto(R,\boxempty)
\]
from the category of racks to the category of quandles, i.e.~a morphism
\[
\Quan\overset{\boxempty}{\longrightarrow}\Rack
\]
of algebraic theories. It is easier to define a morphism
\[
\Rack\overset{\supset}{\longrightarrow}\Quan
\]
of algebraic theories in the other direction: just take the inclusion as the right-adjoint. We then have:

\begin{proposition}\label{prop:q_retract_of_r}
The algebraic theory of quandles is a retract of the algebraic theory of racks: There are morphisms
\[
\Quan\overset{\boxempty}{\longrightarrow}\Rack\overset{\supset}{\longrightarrow}\Quan
\]
such that the composition is the identity.
\end{proposition}

\begin{proof}
By Remark~\ref{rem:comp}, the composition of the right-adjoints is the identity. Passage to left adjoints gives the result.
\end{proof}

I understand that V.~Lebed and L.~Vendramin have, independently from this, obtained results as in Sections 1 and 3 for biracks and biquandles.

%%%

\section{Centers}\label{sec:centers}

The {\it center} of a category~$\bfC$ is defined to be the (abelian) monoid of natural transformations~$\Id_\bfC\to\Id_\bfC$. The elements are families~$\Phi=(\Phi_C)$ of endomorphisms~\hbox{$\Phi_C\colon C\to C$}, one for each object~$C$, such that~$\phi\Phi_C=\Phi_D\phi$ for each morphism~$\phi\colon C\to D$ in~$\bfC$. Multiplication is given by object-wise composition.

\begin{example}
It is easy to see that the center of the category~$\Sets$ of sets is trivial: Its one and only element is the family~$\id=(\id_S)$ of identities.
\end{example}

\begin{example}\label{ex:det}
The center of the category~$\Perm$ of permutations is isomorphic to the group~$\bbZ$ of integers. Clearly, given any permutation~$(S,f)$, it comes with natural self-maps~$S\to S$ that commute with~$f$: the powers~$f^m$ of~$f$! Conversely, given natural self-maps~$\Phi(S,f)$, and an element~$s$ of~$S$, there is a commutative diagram
\[
\xymatrix@C=4em{
\bbZ\ar[d]_s\ar[r]^-{\Phi(\bbZ,+1)}&\bbZ\ar[d]^s\\
S\ar[r]_{\Phi(S,f)}&S
}
\]
when we write~$s\colon \bbZ\to S$ for the unique morphism from the free permutation on one generator~$0$ to~$S$ that sends the generator to~$s$, and hence~$m$ to~$f^m(s)$. This diagram shows that~$\Phi(S,f)$ sends~$s$ to~$f^m(s)$ if~$m$ is the image of~$0$ under~$\Phi(\bbZ,+1)$. It follows that every element in the center is actually given by one of the powers. 
\end{example}

\begin{remark}
More generally, the center of the category of~$M$-sets (Remark~\ref{rem:actions}) is given by the center of the monoid~$M$. In the preceding example, the monoid~$M=\bbZ$ is abelian, so that it agrees with its center. 
\end{remark}

%%%

We can now turn our attention to racks.

\begin{theorem}\label{thm:center}
The center of the category of racks is the free cyclic group that is generated by the canonical automorphism.
\end{theorem}

\begin{proof}
The canonical automorphisms defines an element in the center of the category of racks: The computation 
\[
\rmF_S\phi(x)=\phi(x)\rhd\phi(x)=\phi(x\rhd x)=\phi\rmF_R(x)
\]
implies that the diagram
\[
\xymatrix{
R\ar[r]^\phi\ar[d]_{\rmF_R} & S\ar[d]^{\rmF_S}\\
R\ar[r]_\phi & S
}
\]
commutes, and we have naturality for all morphisms~$\phi\colon R\to S$ of racks.

The canonical automorphism generates a cyclic subgroup of the center. To see that this subgroup is in fact the entire center, we first note that, by naturality, any element in the center is determined by what it does on the free rack on one generator. (See the argument in Example~\ref{ex:det}.) But the free rack on one generator can be modeled as the set~$\bbZ$ together with the rack structure~$a\rhd b = b+1$. Every element can be chosen as a generator. Therefore, every endomorphism is invertible, being given by~$b\mapsto b+n$ for a unique integer~$n$. And then this automorphism is the~$n$-th power of the canonical automorphism. It follows that the center is cyclic. This argument also shows that the order is infinite, so that the cyclic group is free.
\end{proof}

\begin{theorem}
The center of the category~$\Quan$ of quandles is trivial.
\end{theorem}

\begin{proof}
This is a similar argument as in the proof before. The difference is that the free quandle on one generator has a unique element.
\end{proof}

\begin{example}
For comparison, the center of the category of groups is isomorphic to the monoid~$\{\,0,\,1\,\}$ under multiplication. (See~\cite[Prop.~4.2]{Meir+Szymik}, for instance.) The element~$1$ corresponds to the identity~$G\to G$, and the element~$0$ corresponds to the constant homomorphism~$G\to G$.
\end{example}

\begin{remark}
There are very good reasons to study racks and quandles that have some topological structure~(see~\cite{Rubinsztein} and~\cite{Elhamdadi+Moutuou}). In such contexts, the notion of a center of a category as defined here is typically too rigid to be meaningful. The papers~\cite{Szymik} and~\cite{Dwyer+Szymik} develop a suitable `derived' replacement, and show how to reduce the necessary computations in the case of racks and quandles to the ones done in this section.
\end{remark}

%%%

\section{Power operations}\label{sec:powers}

Elements in the center of a category are endomorphisms of the identity functor. In this section, we will see that--for the categories of our present interest--there are many other interesting endofunctors besides the identity functor: power operations. Before we introduce these for racks and quandles, let us briefly review power operations for the easier and better-understood context of permutations.

Let~$(S,f)$ be a permutation in the sense of Section~\ref{sec:perm}. For any given integer~$m$ there is a functor~$\Psi^m\colon\Perm\to\Perm$ that is given on objects by
\[
\Psi^m(S,f)=(S,f^m).
\]
For instance, we have~$\Psi^1(S,f)=(S,f)$, and~$\Psi^0(S,f)=(S,\id)$. In general, we have~$\Psi^m\Psi^n=\Psi^{mn}$. 

\begin{remark}
These equations can be rephrased to say that the multiplicative monoid of integers acts on the theory of permutations.
\end{remark}

The functors~$\Psi^m$ preserve the underlying sets, so that they are right adjoints and have left adjoints.

We can now present the power operations on the category of racks. 

\vbox{\begin{theorem}\label{thm:powerops}
For all integers $m$ there are endomorphisms
\[
\Psi^m\colon\Rack\to\Rack
\]
and
\[
\Psi^m\colon\Quan\to\Quan
\]
of the theories of racks and quandles such that~$\Psi^1=\Id$ and~$\Psi^m\Psi^n=\Psi^{mn}$. \end{theorem}}

\begin{proof}
Let~$(R,\rhd)$ be a rack. For any given integer~$m$ there is a functor~$\Psi^m$ that is given on objects by
\[
\Psi^m(R,\rhd)=(R,\rhd^m),
\]
where~$x\rhd^m y=\ell_x^m(y)$ if again~$\ell_x$ denotes the permutation~$y\mapsto x\rhd y$ of~$R$. The functors~$\Psi^m$ preserve the underlying sets, so that they are right adjoints and have left adjoints. In other words, they give rise to morphisms of algebraic theories. 

If~$(R,\rhd)$ is a quandle, so is~$\Psi^m(R,\rhd)$. Therefore, the theory of quandles is invariant under these operations.
\end{proof}

\begin{remark}
The binary operation $\rhd$ of a rack $R$ defines an invertible element in Przytycki's monoid of magma structures on the set $R$, and the binary operations $\rhd^m$ for~\hbox{$m\in\bbZ$} are the powers of $\rhd$ in that monoid. (See~\cite{Przytycki+Putyra} and~\cite{Przytycki+Sikora}, for instance.) This certainly justifies the notation.
\end{remark}

The rack~$\Psi^0(R,\rhd)$ is a quandle: the set~$R$ together with the trivial structure given by the projection~$x\rhd y = y$. 

\vbox{\begin{proposition}
The operations on racks and quandles are compatible with each other and with the operations on permutations in the sense that the diagrams
\[
\xymatrix{
\Perm\ar[r]^\can\ar[d]^-{\Psi^m}&\Rack\ar[r]^\per\ar[d]^-{\Psi^m}&\Perm\ar[d]^-{\Psi^m}\\
\Perm\ar[r]^\can&\Rack\ar[r]^\per&\Perm
}
\]
and 
\[
\xymatrix{
\Quan\ar[r]^\boxempty\ar[d]^-{\Psi^m}&\Rack\ar[r]^\supset\ar[d]^-{\Psi^m}&\Quan\ar[d]^-{\Psi^m}\\
\Quan\ar[r]^\boxempty&\Rack\ar[r]^\supset&\Quan
}
\]
commute.
\end{proposition}}

\begin{proof}
For the first diagram, we check commutativity of the squares for the right adjoints, and this is easy. For the square to the left, it follows from
\[
\rmF_{(\rhd^m)}=(\rmF_\rhd)^m,
\]
and for the square to the right, it results from
\[
\rhd_{(f^m)}=(\rhd_f)^m.
\]

For the second diagram, we also check commutativity of the squares for the right adjoints. For the square to the right, this is really trivial. For the square to the left, it follows directly from~\eqref{eq:also}. 
\end{proof}

\begin{remark}\label{rem:Kan}
The reader may wonder if there are more endofunctors of the category of racks (or quandles) that commute with the forgetful functor other than the power operations. Since the forgetful functor is representable by the free rack~(or free quandle) on one generator, these correspond to co-rack (or co-quandle) structures on the free rack (or free quandle) on one generator, and the structure morphisms are determined by the image of the generator in the free rack (or free quandle) on two generators. For instance, the identity functor corresponds to the element~$x_1\rhd x_2$ in the free rack (or free quandle) on two generators~$x_1$ and~$x_2$, whereas the power operation~$\Psi^m$ corresponds to~$x_1\rhd^m x_2$. From this point of view, it seems a rather messy endeavor to determine which other elements might give rise to additional operations, and we will content ourselves here with the interesting operations that we have. The analogous question for the theory of groups was answered by Kan~\cite{Kan}; there are no such operations on groups other than the identity.
\end{remark}

%%%

\section{Universal properties}\label{sec:universal}

In this section we will give precise categorical characterizations, in the form of universal properties, of the theories of quandles, involutary racks, and kei that only involve the most general category of racks, the permutations together with their relation to racks via the canonical automorphism, and power operations.

%%%

\subsection{Quandles}

Here is a characterization of the algebraic theory of quandles by means of a universal property. It improves upon the splitting results of Sections~\ref{sec:perm} and~\ref{sec:quan} in that it accounts for the `difference' as well: The difference between the theories of quandles and racks is the same as the difference between the theory of sets and the theory of permutations. More precisely:

\vbox{\begin{theorem}\label{thm:char_quan}
There is a pushout square
\[
\xymatrix{
\Perm\ar[r]^\can\ar[d]_\id&\Rack\ar[d]^\supset&\\
\Sets\ar[r]&\Quan
}
\]
of algebraic theories.
\end{theorem}}

By Propositions~\ref{prop:for_splitting_R},~\ref{prop:Q_splits_off_S}, and~\ref{prop:q_retract_of_r}, all of the arrows in that diagram are actually split.

\begin{proof}
We have to show that a morphism~$\Quan\to\mathbf{T}$ of algebraic theories is the same as two morphisms~$\Sets\to\mathbf{T}$ and~$\Rack\to\mathbf{T}$ whose restrictions to~$\Perm$ agree. Morphisms point in the directions of the left adjoints, and the statement follows from the corresponding statement about the right-adjoints: A `forgetful' functor~$\mathbf{T}\to\Quan$ of is the same as two `forgetful' functors~\hbox{$\mathbf{T}\to\Sets$} and~\hbox{$\mathbf{T}\to\Rack$} whose compositions to~$\Perm$ agree. Now this is rephrasing the fact that a quandle is a rack such that the canonical automorphism is the identity on the underlying set.
\end{proof}

%%%

\subsection{Involutary racks and kei}

Let us recall the following terminology.

\begin{definition}
A rack is {\em involutary} if~$x\rhd(x\rhd y)=y$ for all~$x$ and~$y$. A {\em kei} is an involutary quandle.
\end{definition}

We can now characterize the theories~$\Invo$ of involutary racks and~$\Kei$ by universal properties. This involves the power operations introduced in the preceding Section~\ref{sec:powers}.

\vbox{\begin{theorem}\label{thm:char2}
There are pushout squares
\[
\xymatrix{
\Rack\ar[r]^-{\Psi^2}\ar[d]_\id&\Rack\ar[d]^\supset&\Quan\ar[r]^-{\Psi^2}\ar[d]_\id&\Quan\ar[d]^\supset\\
\Sets\ar[r]&\Invo&\Sets\ar[r]&\Kei
}
\]
of algebraic theories.
\end{theorem}}

\begin{proof}
The general form of the argument is similar to that for Theorem~\ref{thm:char_quan}. Here we are  using the fact that a rack~$(R,\rhd)$ is involutary if and only if we have~\hbox{$\ell_x^2=\id_R$} for all~$x$, and this is the case if and only if~$\Psi^2(R,\rhd)$ is the trivial rack on~$R$. The argument for kei is analogous, of course.
\end{proof}

\begin{remark}
The theorem makes it clear that the theories of quandles and kei are just two terms of an entire sequence
\begin{align*}
\Quan(0)&=\Quan\\
\Quan(1)&=\Sets\\
\Quan(2)&=\Kei\\
\Quan(3)&=\dots
\end{align*}
of theories, where the algebraic theory $\Quan(m)$ is the pushout of the retraction~\hbox{$\Quan\to\Sets$} along the power operation $\Psi^m$ on $\Quan$ as in Theorem~\ref{thm:char2}. Note that the theories~$\Quan(-m)$ and $\Quan(+m)$ are equivalent, so that we only need to list one of them. A similar remark applies to racks, of course.
\end{remark}

We end by spelling out the results about involutary racks and kei that are analogous to the corresponding results for racks and quandles proved earlier in the text.

\begin{proposition}
There are morphisms
\begin{gather*}
\bbZ/2\text{-}\Sets\overset{\can}{\xrightarrow{\hspace*{2em}}}\Invo\overset{\per}{\xrightarrow{\hspace*{2em}}}\bbZ/2\text{-}\Sets\\
\Sets\longrightarrow\Kei\longrightarrow\Sets\\
\Kei\overset{\boxempty}{\longrightarrow}\Invo\overset{\supset}{\longrightarrow}\Kei
\end{gather*}
of algebraic theories such that the compositions are the identities.
\end{proposition}

The proof is as for Propositions~\ref{prop:for_splitting_R},~\ref{prop:Q_splits_off_S}, and~\ref{prop:q_retract_of_r}. 

\vbox{\begin{theorem}\label{thm:char_kei}
There is a pushout square
\[
\xymatrix{
\bbZ/2\text{-}\Sets\ar[r]^-\can\ar[d]_\id&\Invo\ar[d]^\supset&\\
\Sets\ar[r]&\Kei
}
\]
of algebraic theories.
\end{theorem}}

The proof is as for Theorem~\ref{thm:char_quan}. The preceding proposition implies that all of the arrows in the pushout diagram split.

%%%

\section*{Acknowledgements}

I thank V. Lebed, M. Saito, and F. Wagemann for commenting on a draft version of this paper.

%%%

%%%

\vfill

\parbox{\linewidth}{%
Department of Mathematical Sciences\\
NTNU Norwegian University of Science and Technology\\
7491 Trondheim\\
NORWAY\\
\phantom{ }\\
\href{mailto:markus.szymik@math.ntnu.no}{markus.szymik@math.ntnu.no}\\
\href{http://www.math.ntnu.no/~markussz}{www.math.ntnu.no/$\sim$markussz}}

%%%

\end{document}